\documentclass{amsart}

\usepackage{lineno,hyperref}
\usepackage[english]{babel}
\usepackage[utf8]{inputenc}
\usepackage{amsmath}
\usepackage{easybmat}

\usepackage{graphicx}
\usepackage[colorinlistoftodos]{todonotes}
\usepackage{mathrsfs}
\usepackage{yfonts}
\usepackage{eufrak}
\usepackage{bbm}
\usepackage{amssymb}                 
\usepackage{amsfonts}
\usepackage{amsmath}
\usepackage{euscript}
\usepackage{amsthm}
\usepackage{yhmath}
\usepackage{epsfig}
\usepackage{environ}
\usepackage{caption}
\usepackage{indentfirst}
\usepackage[usenames,dvipsnames]{pstricks}
\usepackage{pst-grad}
\usepackage{pst-plot}
\usepackage{tikz}
\usepackage{mathtools}

\DeclarePairedDelimiter\floor{\lfloor}{\rfloor}
\usetikzlibrary{quotes,angles,positioning}
\usetikzlibrary{decorations.shapes}
\usetikzlibrary{calc}
\usetikzlibrary{arrows.meta}

\chardef\bslash=`\\ 

\newtheorem{thm}{Theorem}[section]

\newtheorem*{prop_A}{Proposition A}

\newtheorem*{theorem_1}{Theorem 1}
\newtheorem*{theorem_2}{Theorem 2}
\newtheorem*{theorem_1.3}{Theorem 1.3}
\newtheorem*{theorem_1.4}{Theorem 1.4}

\newtheorem{lem}[thm]{Lemma}
\newtheorem{prop}[thm]{Proposition}
\newtheorem{remark}{Remark}[section]

\newtheorem{example}{Example}[section]

\theoremstyle{definition}
\newtheorem{defn}{Definition}[section]

\newtheorem{notation}{Notation}
\newtheorem*{observation}{Observation}
\theoremstyle{remark}

\DeclareMathOperator{\aut}{\mathrm{Aut}}
\DeclareMathOperator{\Hom}{Hom}
\DeclareMathOperator{\gl}{GL}

\DeclareMathOperator{\stab}{Stab}
\DeclareMathOperator{\der}{Der}
\DeclareMathOperator{\dime}{dim}

\DeclareMathOperator{\rk}{\mathrm{rk}}

\DeclareMathOperator{\End}{End}

\DeclareMathOperator{\chc}{\mathrm{char}}

\DeclareMathOperator{\supp}{\mathrm{Supp}}

\makeatletter
\newcommand{\extp}{\@ifnextchar^\@extp{\@extp^{\,}}}
\def\@extp^#1{\mathop{\bigwedge\nolimits^{\!#1}}}
\makeatother

\NewEnviron{myequation}{%
\begin{equation*}
\scalebox{0.8}{$\BODY$}
\end{equation*}
}
\modulolinenumbers[5]
\usetikzlibrary{patterns}

\newlength{\hatchspread}
\newlength{\hatchthickness}
\newlength{\hatchshift}
\newcommand{\hatchcolor}{}
\tikzset{hatchspread/.code={\setlength{\hatchspread}{#1}},
         hatchthickness/.code={\setlength{\hatchthickness}{#1}},
         hatchshift/.code={\setlength{\hatchshift}{#1}},
         hatchcolor/.code={\renewcommand{\hatchcolor}{#1}}}
\tikzset{hatchspread=3pt,
         hatchthickness=0.4pt,
         hatchshift=0pt,
         hatchcolor=black}
\pgfdeclarepatternformonly[\hatchspread,\hatchthickness,\hatchshift,\hatchcolor]
   {custom north west lines}
   {\pgfqpoint{\dimexpr-2\hatchthickness}{\dimexpr-2\hatchthickness}}
   {\pgfqpoint{\dimexpr\hatchspread+2\hatchthickness}{\dimexpr\hatchspread+2\hatchthickness}}
   {\pgfqpoint{\dimexpr\hatchspread}{\dimexpr\hatchspread}}
   {
    \pgfsetlinewidth{\hatchthickness}
    \pgfpathmoveto{\pgfqpoint{0pt}{\dimexpr\hatchspread+\hatchshift}}
    \pgfpathlineto{\pgfqpoint{\dimexpr\hatchspread+0.15pt+\hatchshift}{-0.15pt}}
    \ifdim \hatchshift > 0pt
      \pgfpathmoveto{\pgfqpoint{0pt}{\hatchshift}}
      \pgfpathlineto{\pgfqpoint{\dimexpr0.15pt+\hatchshift}{-0.15pt}}
    \fi
    \pgfsetstrokecolor{\hatchcolor}
    \pgfusepath{stroke}
   }

\usepackage[a4paper,top=3cm,bottom=2cm,left=3cm,right=3cm,marginparwidth=1.75cm]{geometry}

\usepackage{amsmath}
\usepackage{graphicx}
\usetikzlibrary{patterns}
\usepackage[colorinlistoftodos]{todonotes}

\newcommand{\eval}[2][\right]{\relax
  \ifx#1\right\relax \left.\fi#2#1\rvert}

\let\abs=\envert

\begin{document}
\title{Triviality of the automorphism group of the multiparameter quantum affine $n$-space}
\author[A. Gupta and S. Mandal]{Ashish Gupta and Sugata Mandal}

\address{Ashish Gupta, Department of Mathematics\\
Ramakrishna Mission Vivekananda Educational and Research Institute (Belur Campus) \\
Howrah, WB 711202\\
India}
\email{a0gupt@gmail.com \thanks{}}
\address{Sugata Mandal, Department of Mathematics\\
Ramakrishna Mission Vivekananda Educational and Research Institute (Belur Campus) \\
Howrah, WB 711202\\
India}
\email{gmandal1961@gmail.com\thanks{}}


\begin{abstract}
A multiparameter quantum affine space of rank $n$ is the $\mathbb F$-algebra generated by indeterminates $X_1, \cdots, X_n$ satisfying $X_iX_j = q_{ij} X_jX_i \ (1 \le i < j \le n)$ where $q_{ij}$ are nonzero scalars in $\mathbb F^\ast$.  
The corresponding quantum torus is generated by the $X_i$ and together with their inverses subject to the same relations. So far the automorphisms of a quantum affine space have been considered mainly in the uniparameter case, that is, $q_{ij} = q$. We remove this restriction here.

Necessary and sufficient conditions are obtained for the quantum affine space to be rigid, that is, the only automorphisms are the trivial ones arising from the action of the torus $(\mathbb F^\ast)^n$. These conditions are based on the multiparameters $q_{ij}$ and also on the subgroup of $\mathbb F^\ast$ 
generated by these multiparameters. 

We employ the results in J. Alev and M. Chamarie,
  Derivations et automorphismes de quelques algebras quantiques,
 Communications in Algebra, 1992 (20), 1787-1802, and point out a small error in a main theorem in this paper which however remains valid with a small modification. 
 
 We also note that a quantum affine space whose corresponding quantum torus has dimension one necessarily has a trivial automorphism group. This is a consequence of a result of J.~M.~Osborne,  D.~S.~Passman, Derivations of Skew Polynomial Rings, J. Algebra, 1995, 176, 417--448. We expand the known list of examples of quantum tori that have dimension one and are thus hereditary noetherian domains.


\noindent 
\textbf{2010 Math. Subj. Class.}: 16S38; 16S35; 16S36; 16W20
\end{abstract}

\maketitle
\tableofcontents

\section{Introduction}\label{intro}

Quantum affine spaces and their localizations (known as quantum tori) are known to  
 play a key role in the theory of quantum groups~\cite{ART1997, BrGo} and also in non-commutative geometry~\cite{Man}. The quantum tori arise also in Lie theory as coordinate structures of extended affine Lie algebras~\cite{NeebKH2008} and in the representation theory of nilpotent groups~\cite{Br2000}. The automorphisms of these algebras have been considered in \cite{AC1992}, \cite{OP1995}, \cite{ART1997} and \cite{NeebKH2008}. The substantial paper \cite{MY12} considers the automorphisms of certain completions of a quantum torus algebra in connection with the automorphisms of quantum enveloping algebras. In this last paper it is noted that the automorphisms of a quantum affine space of certain types extend to bifinte unipotent automorphisms of the completion. The automorphisms of quantum division rings are studied in \cite{VA2000}.

Let us briefly recall the definitions. Let $\mathbb{F}$ be a field and $\mathfrak q =  (q_{ij})$ be a multiplicatively anti-symmetric $n \times n$-matrix with entries in $\mathbb F^\ast$. This means that
$q_{ii} = 1$ and $q_{ji} = q_{ij}^{-1}.$  A (rank-$n$) quantum affine space $\mathcal O_{\mathfrak q} = \mathcal O_{\mathfrak q}(\mathbb F^n)$ over the field $\mathbb F$ can is defined (as a quotient of the free algebra $\mathbb F\{ X_1, X_2, \cdots, X_n \}$) as follows

\begin{align}
\label{quantum_space_def}
\mathcal O_{\mathfrak q} &:= \mathbb F\{ X_1, X_2, \cdots, X_n \}/\langle X_i X_j - q_{ij}X_j X_i \mid 1 \le i < j \le n \rangle, \qquad q_{ij} \in \mathbb F^\ast.
\end{align}

This ring can also be presented as an iterated skew polynomial ring. 
Localizing a quantum affine space $\mathcal O_{\mathfrak q}$ at the multiplicative subset generated the variables $X_1, \cdots, X_n$ yields the Laurent version of this algebra which is known as the rank-$n$ quantum torus. We denote this quantum torus as $\widehat {\mathcal O}_{\mathfrak q}$.
A (rank-$n$) quantum torus can thus be defined as an algebra generated by the variables $X_1, \cdots, X_n$ together with their inverses subject to relations  \[ X_i X_j = q_{ij}X_j X_i.  \]

\begin{defn}\label{lam-grp}
For a quantum affine space $\mathcal O_{\mathfrak q}$ and the quantum torus 
$\widehat{\mathcal O}_{\mathfrak q}$ the $\lambda$-group (denoted $\Lambda$) 
is the subgroup of $\mathbb F^\ast$ generated by the multiparameters $q_{ij}$. 
\end{defn}

For an abelian group $A$ by the \emph{rank} of $A$ we mean its torsion-free rank, that is, \[ \rk(A): = \dim_{\mathbb Q}(A \otimes_{\mathbb Z} \mathbb Q). \]
 
 It was shown in \cite{MP}  that the Krull and the global dimension for the quantum tori coincide and so we will write $\dim(\widehat {\mathcal O}_{\mathfrak q})$ to denote either of these. In the study of the quantum torus this dimension plays an important role~\cite{MP, Br2000}. Let $\aut_{\mathbb F}(\mathcal O_{\mathfrak q})$ denote the $\mathbb F$-automorphism group of the $\mathbb F$-algebra $\mathcal O_{\mathfrak q}$. We show the following.
 \begin{theorem_1}
Suppose that $\chc(\mathbb F) = 0$ and $n \ge 3$. Let $\mathfrak q = (q_{ij})$ be a multiplicatively antisymmetric matrix such that for $ i<j $ atmost one entry $ q_{ij}$ equals to $1$. Then
\begin{itemize}
    \item[(1)] if no $q_{ij}$ equals to  $1$ for all $ i < j$,  then the automorphism group  \[ \aut_{\mathbb F}(\mathcal O_{\mathfrak q}) = {(\mathbb F^\ast)}^n \] if and only if the following conditions hold
    \begin{itemize}
        \item[(a)] $ E= 0$\label{E=01} and
        \item[(b)] there does not exist any non-identity permutation matrix $m_{\sigma} $  satisfying 
        \begin{equation*} \label{commute permutation1}
            \mathfrak{q}\mathfrak{m_{\sigma}}=\mathfrak{m_{\sigma}}\mathfrak{q}.
        \end{equation*}
    \end{itemize}
    \item[(2)] if exactly one entry, say $ q_{i'j'}$ with $i'<j'$ is $1$,  then  the automorphism group  \[ \aut_{\mathbb F}(\mathcal O_{\mathfrak q}) = {(\mathbb F^\ast)}^n \] if and only if the following conditions hold
    \begin{itemize}
        \item[(a)] \label{2nd E=01} $ E= 0$  and
        \item[(b)] there does not exist any permutation $\pi_{1}$ on the set $\{i',j'\} $ and  non-identity  permutation $ \pi_{2} $ on the set $ \mathcal{J}= \{1,2,\cdots,n\} \setminus \{i',j'\}$ satisfying
        \begin{equation*}\label{2nd commute permutation1}
          q_{i'r}=q_{\pi_{1}(i')\pi_{2}(r)}, \quad q_{j'r}=q_{\pi_{1}(j')\pi_{2}(r)},  \qquad \forall r \in \mathcal{J}  
        \end{equation*}
        such that $\mathfrak{q'}\mathfrak{m}_{\pi_{2}}=\mathfrak{m}_{\pi_{2}}\mathfrak{q'} $
          where  $\mathfrak q'$ is a multiplicatively anti-symmetric matrix  obtained from $ \mathfrak{q}$ by removing the $ i',j' $ rows and columns simultaneously.
    \end{itemize}
\end{itemize}

\end{theorem_1}
\begin{theorem_2}
A quantum affine space $\mathcal O_{\mathfrak q} = \mathcal O_{\mathfrak q}(\mathbb F^n) \ (n \ge 3)$ whose $\lambda$-group is torsion-free and has rank no smaller than $c: = \binom{n - 1}{2} + 1$ satisfies  
 \[ \aut_{\mathbb F}(\mathcal O_{\mathfrak q})= (\mathbb{F^{\ast}})^n. \]  
Moreover, $c$ is minimal with respect to this property and  for each $n \ge 3$ and each $r  < \binom{n - 1}{2} + 1$ there exists an algebra $\mathcal O_{\mathfrak q}$ with $\lambda$-group equal to $\mathbb Z^r$ but whose automorphism group embeds $ (\mathbb F^\ast)^n \rtimes \mathbb F^+$. 
\end{theorem_2}
This paper is organized as follows: Section 2 treats the facts concerning the quantum torus that we will be needing. Sections 3 and 4 are devoted to the proofs of Theorems 1 and Theorem 2 respectively.  Our final section addresses the question of obtaining quantum tori of rank $n$ with dimension one as the automorphism group will be always trivial in this case.

\section{The Quantum Torus}

\label{aut-qtor}
\subsection{Twisted Group Algebra Structure} 
Let $\Gamma : = \mathbb Z^n$. A twisted group algebra $\mathbb F \ast \Gamma$ is an $\mathbb F$-algebra which has a copy $\bar \Gamma : = \{\bar \gamma \mid \gamma \in \Gamma \}$ for an $\mathbb F$-basis satisfying 
\[ \bar \gamma \bar {\gamma'} = f(\gamma, \gamma')\overline{\gamma \gamma'}, \qquad \gamma, \gamma' \in \Gamma. \]  
for a suitable $2$-cocycle $f: \Gamma \times \Gamma \rightarrow \mathbb F^\ast$. A rank-$n$ quantum torus $\widehat {\mathcal O}_{\mathfrak q}$ has such a structure. Indeed if we define a map $\Gamma \rightarrow \widehat {\mathcal O}_{\mathfrak q}$ via \[ (\gamma_1, \cdots, \gamma_n) = :\gamma \mapsto \mathbf X^{\gamma} := X_1^{\gamma_1}\cdots X_n^{\gamma_n} \] 
then it can be checked that the above conditions are satisfied. 

Clearly each element $\alpha$ of a twisted group algebra $\mathbb F \ast \Gamma$ may be expressed as a finite sum $\alpha = \sum_{\gamma \in \Gamma} a_\gamma \bar \gamma\ (a_\gamma \in \mathbb F)$. The subset of $\gamma \in \Gamma$ such that $a_\gamma \ne 0$ is called the support of $\alpha$ and denoted as $\supp(\alpha)$. For a subgroup $B \le \Gamma$ the subset 
\[  \{ \alpha \in \mathbb F \ast \Gamma \mid \supp(\alpha) \subseteq B \} \]   
is a twisted group algebra of $B$ over $\mathbb F$ and is denoted as $\mathbb F \ast B$. \label{sec2.2} 
\subsection{The commutator map $\lambda$}  
 In a quantum torus the monomials 
$\mathbf X^{\gamma} := X_1^{\gamma_1}\cdots X_n^{\gamma_n}$ are units and the group-theoretic commutator $[\mathbf X^{\gamma}, \mathbf X^ {\gamma'}]$ defined as 
\[ [\mathbf X^{\gamma}, \mathbf X^ {\gamma'}] := \mathbf X^{\gamma} \mathbf X^{\gamma'}{(\mathbf X^{\gamma})}^{-1}{(\mathbf X^{\gamma'})}^{-1} \] 
yields an alternating bi-character ((e.g., \cite[Section 1]{OP1995})) \begin{equation}
\label{lmbdadefn}    
\lambda: \Gamma  \times \Gamma  \rightarrow \mathbb F^\ast,  \qquad  \lambda(\gamma, \gamma') = [\mathbf X^{\gamma},\mathbf X^{\gamma'}], \qquad   \mathbf \gamma, \gamma'  \in \Gamma. 
\end{equation} 
In particular,  
\[ \lambda(\mathbf e_i, \mathbf e_j) = [X_i, X_j] = q_{ij},  \qquad \forall 1 \le i,j \le n, \] 
where $\mathbf e_1, \cdots \mathbf e_n$ are the standard basis vectors of the $\mathbb Z$-module $\Gamma$.
\subsection{Dimension} 
\label{dim-of-qtorus}
As noted above a quantum torus $\widehat{O}_{\mathfrak q}$ is a twisted group algebra $\mathbb F \ast \Gamma$. It was shown in \cite[Theorem A]{Br2000}
that the dimension of a quantum torus equals the supremum of the ranks of subgroups $B \le \Gamma$ such that the subalgebra $\mathbb F \ast B$ is commutative (Note that $\mathbb F \ast C$ is commutative for any cyclic subgroup $C \le \Gamma$). It follows that $\dim(\widehat{O}_{\mathfrak q})$ equals the cardinality of a maximal system of independent commuting monomials in $\widehat{O}_{\mathfrak q}$.

\subsection{The automorphism group of a quantum torus}
 It is known (e.g., \cite{MP}) that the units of a quantum torus algebra are trivial, that is, are of the form $a \mathbf X^{\gamma}$, where $a \in \mathbb F^\ast$. 
By $\aut_{\mathbb F}(\widehat{\mathcal O}_{\mathfrak q})$ we denote the group of all 
 $\mathbb F$-automorphisms of the quantum torus  $\widehat {\mathcal O}_{\mathfrak q}$.  
It is easily seen (e.g., \cite{OP1995}) that the action of the group $\aut_{\mathbb F}(\widehat{\mathcal O}_{\mathfrak q})$  on the quantum torus $\widehat{\mathcal O}_{\mathfrak q}$ induces an action of this same group on the group $\mathscr U$ of trivial units fixing $\mathbb F^\ast$ element-wise.  There is thus an action of this same group $\aut_{\mathbb F}(\widehat{\mathcal O}_{\mathfrak q})$ on the quotient  $\mathscr U/\mathbb F^\ast \cong \Gamma$ yielding a homomorphism 
\begin{equation}\label{actn-Gma}   \aut_{\mathbb F}(\widehat{\mathcal O}_{\mathfrak q}) \longrightarrow \aut \Gamma = \gl(n, \mathbb Z) \end{equation} 
whose kernel is the group 
of all \emph{scalar automorphisms} defined by $\psi(\mathbf X^{\gamma}) = \phi(\gamma)(\mathbf X^{\gamma})$ for $\phi \in  \Hom(\Gamma, \mathbb F^\ast)$~\cite{OP1995}. Clearly  this kernel can be identified  with the algebraic torus  $(\mathbb F^\ast)^n$.

By \cite[Lemma 3.3(iii)]{OP1995} the image (in $\gl(n, \mathbb Z)$)
of the map in \eqref{actn-Gma} is the subgroup of all  $\sigma \in  \gl(n, \mathbb Z)$ such that 
  
\begin{equation}
\label{form_prsvng}
\lambda (\sigma \gamma, \sigma \gamma') = \lambda(\gamma, \gamma' ) \qquad\ \ \  \forall \gamma, \gamma' \in \Gamma.
\end{equation}
This subgroup is denoted $\aut(\mathbb Z^n, \lambda)$ is known as the \emph{nonscalar automorphism group}.
By the foregoing discussion we obtain the following exact sequence noted in \cite{NeebKH2008}:
\begin{equation}\label{Neeb-Key-Result1} 
1 \rightarrow  (\mathbb F^\ast)^n \rightarrow \aut_{\mathbb F}(\widehat{\mathcal O}_{\mathfrak q})  \rightarrow \aut(\mathbb Z^n, \lambda) \rightarrow 1. \end{equation}

\section{Proof of Theorem 1}  

\begin{defn}[Section 1.4 of \cite{AC1992}]
An automorphism $\sigma$ of $\mathcal O_{\mathfrak q}$ is called linear if it has the form
\begin{equation}\label{lin_aut_def} 
\sigma(X_i)= \sum_{j = 1}^n \alpha_{ij}X_j \qquad \forall i \in \{1, \cdots, n\}, \qquad (\alpha_{ij}) \in \gl(n, \mathbb F).
\end{equation}
\end{defn}
 
For a matrix $(\alpha_{ij}) \in \gl(n, \mathbb F)$ to define an automorphism as in \eqref{lin_aut_def} the   
following necessary and sufficient conditions must hold (\cite{AC1992}):
\begin{equation}\label{AC-cond-lin-aut}
\alpha_{ik} \alpha_{jl}(1 - q_{ij} q_{lk}) = \alpha_{il} \alpha_{jk}(q_{ij} - q_{lk})  \qquad\ \ \forall i < j , \ \ \forall k \le l.  
  \end{equation}
The last equation may be re-written as 
\begin{equation}\label{AC-cond-lin-aut-re}
   \alpha_{ik} \alpha_{jl}(q_{kl} - q_{ij}) = \alpha_{il} \alpha_{jk}(q_{kl}q_{ij} - 1)  \  \qquad\forall i < j , \ \ \forall k \le l. 
\end{equation}
Setting $k = l$ in the last equation we obtain 
\begin{equation}\label{setkeql}
\alpha_{ik} \alpha_{jk} (1 - q_{ij}) = \alpha_{ik} \alpha_{jk}(q_{ij} - 1)    \qquad \ \forall i < j, \ \ \forall k \in \{1, \cdots, n\} . 
\end{equation}

\begin{observation}\label{avd-1}
Clearly, the last equation means that if $\chc(\mathbb F) \ne 2$ and none of the multiparameters $q_{ij}\  (i < j)$ equals to unity  
then at least one of the coefficients $\alpha_{ik}$ and $\alpha_{jk}$ vanishes.
It is immediate that in this case the nonsingular matrix $(\alpha_{ij})$ has exactly one nonzero entry in each row and each column. 
\end{observation}

The next proposition is an easy consequence of the preceding observation.

\begin{prop}\label{simple-cse-q(ij)-not-1}
Suppose that $\chc(\mathbb F) \ne 2$ and the entries of $\mathfrak q$ satisfy \begin{equation*} \label{sim-cond1}
q_{ij} \ne 1,\qquad \qquad \forall 1 \le i < j \le n.  
\end{equation*}
Then each linear automorphism $ \sigma $ of $\mathcal O_{\mathfrak q}$ has the form $ X_i \rightarrow k_{i} X_{\sigma(i)}$ for a suitable permutation $\sigma$ of the variables $ X_{i}$ where  $k_{i} \in \mathbb F^\ast $.
Thus, 
\begin{equation}\label{lin_aut}
\aut_{\mathrm L}(\mathcal O_{\mathfrak q}) \cong ({\mathbb F^\ast})^n \rtimes \mathcal{P}  
\end{equation}
for a subgroup $\mathcal P $ of $S_n$.
\end{prop}
For a permutation $ \sigma $ we denote by $\mathfrak{m_{\sigma}}$ the corresponding permutation matrix.
\begin{prop}\emph{(\cite[Remark 3.2]{OP1995})}\label{rem}
 Suppose that $ n \ge 3$. Let $\mathfrak q = (q_{ij})$ be a multiplicatively antisymmetric matrix. For a permutation $\sigma$ the corresponding change of variables $x_{i} \rightarrow x_{\sigma(i)}$ is an automorphism of   $\mathcal O_{\mathfrak q} $  if and only if 
 \begin{equation} \label{automorphism condition}
  q_{ij} = q_{\sigma(i) \sigma(j)}, \qquad \qquad \forall 1 \le i < j \le n.
 \end{equation}
 Equivalently,
 \[\mathfrak{q}\mathfrak{m_{\sigma}}=\mathfrak{m_{\sigma}}\mathfrak{q}.\]

\end{prop}
\begin{remark}
    In the situation of Proposition \ref{simple-cse-q(ij)-not-1} if $ q_{ij} \neq -1$ and $\sigma $ is a non-identity permutation in $ \mathcal{P}$ and $\sigma$ is decomposed into cycles then any cycle in this decomposition of $ \sigma $ must have odd length. Indeed if $ (i_1 i_2)$ is a 2-cycle in the decomposition of $\sigma$ then in view of  \ref{automorphism condition} we have $({q_{i_1 i_2}})^{2} = 1$. Again if $(i_1i_2i_3 \cdots i_m)$ is an $m$-cycle ($m \geq 4)$ where $m$ is even then by \eqref{automorphism condition} we have 
\begin{equation*}
    (q_{i_{m/2},i_{m}})^{-1} = 
 q_{i_{1},i_{m/2+1}}=q_{i_{2},i_{m/2 +2}} = \cdots = 
    q_{i_{m/2-1},i_{m/2+m/2-1}} =q_{i_{m/2},i_{m}}
\end{equation*}
whence $(q_{i_{m/2},i_{m}})^2=1 $. 
\end{remark}
 Let $\der(\mathcal O_{\mathfrak q})$ denote the module of derivations of  $\mathcal O_{\mathfrak q}$.
 We recall the submodule  $ E $  of $\der(\mathcal O_{\mathfrak q})$ in \cite{AC1992}.
Suppose \[\Lambda_{i} = \{ \nu \in {\mathbb{N}}^{n} :  \nu_{i} = 0~~\text{and}~~\prod_{k} q_{kj}^{{\nu}_{k}} = q_{ij}~~ \text{for all}~~ j ~~ \text{such that}~~  j\neq i \}.\]
 For all $ \nu \in \Lambda_{i}$ there exists a derivation $ D_{i\nu}$ of $\mathcal O_{\mathfrak q}$  defined by \[ D_{i \nu}(x_{j})= \delta_{ij} {\bf{x}}^{\nu}\] for all $ j $. Note that $D_{i \nu}^{2} = 0$. We then  define  \[E = \oplus_{i}(\oplus_{\nu \in \Lambda_{i}} \mathcal O_{\mathfrak q} D_{i\nu}). \] 
\begin{lem}\label{equal rows}
Suppose that $\chc(\mathbb F) \ne 2$ and $n \ge 3$. Let $\mathfrak q = (q_{ij})$ be a multiplicatively antisymmetric matrix. Then no two rows of $\mathfrak{q}$ are equal if $ E = 0 $.  
\end{lem}
\begin{proof}
Suppose that  $ i$-th and  the $ j$-th rows of $\mathfrak{q}$ be equal, that is, $ q_{ir} = q_{jr}$ for all $ r\in \{1,2,\cdots,n\}$. Let $ k \neq i $  and $\nu:=(0,\cdots,1_j,\cdots,0) $. We claim that $ \nu \in  \Lambda_{i} $. Indeed  from the  definition of $ \Lambda_{i}$ we must have \[\prod_{p} q_{ip}^{{\nu}_{p}} = q_{ik}.\]  But this is true as $ q_{ik} = q_{jk}$ for all $ k \neq i $. Thus $ E \neq 0 $ - a contradiction. 
\end{proof}
\begin{thm}\emph{(\cite[Proposition 1.4.3]{AC1992})}\label{rect}
Suppose that $\chc(\mathbb F) = 0$ and $n \ge 3$. Let $\mathfrak q = (q_{ij})$ be a multiplicatively antisymmetric matrix such that for all $ i $ there exist  $ j $ such that $ q_{ij} \neq 1 $. If $ E = 0 $ then \[\aut_{\mathbb F}(\mathcal O_{\mathfrak q})=\aut_L(\mathcal O_{\mathfrak q}) \]
\end{thm}
The above theorem is a rectification of the proposition of \cite[Section 1.4.3]{AC1992}.
There is a small error in the proof of this proposition. They used the fact that $ E = 0 $ implies for all $ i $ there exist  $ j $ such that $ q_{ij} \neq 1 $. This is not true in general: for  example, consider $ n = 3 $ and $ q_{12}=q_{13}= 1 $, $q_{23} = q$ where $ q $ is not a root of unity. Then clearly $ E = 0 $ but there is an automorphism $ \phi_{b} $ for each $ b \in {\mathbb F}^{\times} $ defined by 
\[\phi_b(X_{i})=  \begin{cases}
                     X_i  &\text{if } i \neq 1\\
                     X_i + b, & \text{if } i=1\end{cases}\]
 Clearly $ \phi_{b} $ is not a linear automorphism.
\begin{theorem_1}
Suppose that $\chc(\mathbb F) = 0$ and $n \ge 3$. Let $\mathfrak q = (q_{ij})$ be a multiplicatively antisymmetric matrix such that for $ i<j $ atmost one entry $ q_{ij}$ equals to $1$. Then
\begin{itemize}
    \item[(1)] if no $q_{ij}$ equals to  $1$ for all $ i < j$,  then the automorphism group  \[ \aut_{\mathbb F}(\mathcal O_{\mathfrak q}) = {(\mathbb F^\ast)}^n \] if and only if the following conditions hold
    \begin{itemize}
        \item[(a)] $ E= 0$\label{E=0} and
        \item[(b)] there does not exist any non-identity permutation matrix $m_{\sigma} $  satisfying 
        \begin{equation} \label{commute permutation}
            \mathfrak{q}\mathfrak{m_{\sigma}}=\mathfrak{m_{\sigma}}\mathfrak{q}.
        \end{equation}
    \end{itemize}
    \item[(2)] if exactly one entry, say $ q_{i'j'}$ with $i'<j'$ is $1$,  then  the automorphism group  \[ \aut_{\mathbb F}(\mathcal O_{\mathfrak q}) = {(\mathbb F^\ast)}^n \] if and only if the following conditions hold
    \begin{itemize}
        \item[(a)] \label{2nd E=0} $ E= 0$  and
        \item[(b)] there does not exist any permutation $\pi_{1}$ on the set $\{i',j'\} $ and  non-identity  permutation $ \pi_{2} $ on the set $ \mathcal{J}= \{1,2,\cdots,n\} \setminus \{i',j'\}$ satisfying
        \begin{equation}\label{2nd commute permutation}
          q_{i'r}=q_{\pi_{1}(i')\pi_{2}(r)}, \quad q_{j'r}=q_{\pi_{1}(j')\pi_{2}(r)},  \qquad \forall r \in \mathcal{J}  
        \end{equation}
        such that $\mathfrak{q'}\mathfrak{m}_{\pi_{2}}=\mathfrak{m}_{\pi_{2}}\mathfrak{q'} $
          where  $\mathfrak q'$ is a multiplicatively anti-symmetric matrix  obtained from $ \mathfrak{q}$ by removing the $ i',j' $ rows and columns simultaneously.
    \end{itemize}
\end{itemize}
\end{theorem_1}
\begin{proof}
$(1)$ We first  show the necessity.
By  definiition  if  $ E \neq 0 $ then there exists an $i$ such that $\Lambda_{i}\neq 0$. Let $\nu=(\nu_{1},\cdots,\nu_{n})\in \Lambda_{i}$ such that $ \nu \neq 0$. It is easy to check that   the automorphism $\exp(D_{i\nu})$ is not a member of $ \mathbb({F^{*}})^n$ where $D_{i\nu}$ is the locally nilpotent   derivation defined by $D_{i\nu}(X_{j})=\delta_{ij}X^{\nu}$. Suppose  there exist a non-identity permutation $ \sigma$ satisfying \eqref{commute permutation} then by Proposition \ref{rem} the map $ X_{i} \rightarrow X_{\sigma(i)}$ is a nontoric automorphism. 
  Conversely, suppose the conditions in (1) hold.  Now  \[\aut_{\mathbb F}(\mathcal O_{\mathfrak q})=\aut_L(\mathcal O_{\mathfrak q}) \] by the Theorem  \ref{rect}. Also from  Observation \ref{avd-1}  any linear automorphism must be of the form  $ X_{i} \rightarrow k_i X_{\sigma(i)} $ for some permutation $ \sigma $ where $ k_i \in F^{\times} $.  By  Proposition \ref{rem}  $\sigma$ must be the identity permutation. \newline
 $ (2)$ The proof of the necessity of the condition $ E= 0 $ is same as in part $(1)$. Again $(b)$ is necessary as otherwise  the map $ \phi $ defined by 
 \[\phi(X_{i})=  \begin{cases}
                     X_{\pi_{1}(i)}  &\text{if } i \in \{i',j'\}\\
                     X_{\pi_{2}(i)}, & \text{if } i\in\mathcal{J}\end{cases}\]
 is a nontoric automorphism of $\mathcal O_{\mathfrak q}$ by proposition \ref{rem}.
  Conversely, the conditions supposed that 2(a) and 2(b) hold.\\
As seen above in this case, we have     \[\aut_{\mathbb F}(\mathcal O_{\mathfrak q})=\aut_L(\mathcal O_{\mathfrak q}). \]\\
Suppose that \[ A = (\alpha_{ij}) \in \gl(n, \mathbb F)\] 
induces a linear automorphism $ \alpha$ of the given quantum affine space. 
Using \eqref{setkeql} it is easily seen that in any column of $A$ at most two entries can be nonzero and in the case there are two nonzero entries, these must be in the $i'$-th and 
$j'$-th rows. 

We claim that there can be at most two columns in $A$ that have two nonzero entries. Indeed suppose that $2 + s$ columns have two nonzero entries necessarily in the $i'$-th and $j'$-th rows. As just noted the remaining $n - 2 - s$ columns each has exactly one non-zero entry.  Clearly, if $s > 0$ these $n - 2 - s$ nonzero entries in $A$ cannot fulfill the requirement of a non-zero entry in each of the $n - 2$ rows other than the $i'$-th and $j'$-th rows. This shows that $s = 0$. 

Next we note that at least one of the entries $\alpha_{i'i'}$ and  $\alpha_{j'i'}$ is non-zero. To see this we pick $m < p$ in the range $1, \cdots,n$. 
By \eqref{AC-cond-lin-aut-re} we have noting that $q_{i'j'} = 1$  
\begin{align}
\label{apln_fundl_rln1}
\alpha_{i'm}\alpha_{j'p}(q_{mp} - 1) &= \alpha_{i'p}\alpha_{j'm}(q_{mp} - 1). 
\end{align} 
 If $(m,p) \ne (i',j')$ then by the hypothesis on the part $(2)$  $q_{mp} \ne 1$ and \eqref{apln_fundl_rln1} means that  
the minor of $A$ corresponding to the $2 \times 2$ submatrix formed by the $i'$-th and $j'$-th rows and the $m$-th and $p$-th columns is equal to zero. If $\alpha_{i'i'} = \alpha_{j'i'} = 0$ then the minor corresponding to the $2 \times 2$ submatrix $K$ defined by the $i'$-th and $j'$-th rows and the $i'$-th and $j'$-th columns is also equal to zero.
This would mean that a row of the exterior square $\wedge^2 A$ of $A$ is the zero row contradicting the assumption $A$ is non-singular. 

By the same token at least one of the entries $\alpha_{i'j'}$ and $\alpha_{j'j'}$ in column $j'$ is nonzero. 
Moreover, the non-zero entries in the two columns, namely, $i'$ and $j'$ cannot be in only one of the rows $i'$ or $j'$ as in this case the determinant of $K$ will be zero.

We now claim that if a column of $A$ has two non-zero entries then it must be the $i'$-th or the $j'$-th column. Indeed let $h$ be a column of $A$ having two non-zero entries. As noted above these non-zero entries of $h$ must be in rows $i'$ and $j'$. Thus the three columns $i'$, $j'$ and $h$ have non-zero entries only in rows $i'$ and $j'$. Consequently there can be at most $n -3$ nonzero entries in columns other than $i',j'$ and $h$ that are contained in the $n - 2$ rows other than $i'$ and $j'$. But this means there is a row with no non-zero entry contradicting the assumption that $A$ is non-singular.

It follows that $\alpha(X_{i'})$ and $\alpha(X_{j'})$ both lie in the $k$-subspace spanned by $X_{i'}$ and $X_{j'}$
while $\alpha(X_r)=k_{r}X_{\pi_{2}(r)}$ for all $r \in \mathcal{J} $ where $ k_{r}\in \mathbb{F}^{\times}$ and $ \pi_{2}$ is a permutation on the set $ \mathcal{J}$. 
 We claim that either $\alpha(X_{i'}) \in \mathbb F^{\times} X_{i'}$ or $\alpha(X_{i'}) \in \mathbb F^{\times} X_{j'}$. Indeed let \[ \alpha(X_{i'})=aX_{i'}+ bX_{j'} \] where $ a,b \in \mathbb{F}^{\times}$. Applying $\alpha$ to the relations 
 \[  X_{i'}X_{r}= q_{i'r} X_{r}X_{i'}, \qquad \forall r \in \mathcal{J}  \] 
 we obtain \[(aX_{i'}+ bX_{j'})X_{\pi_{2}(r)}=q_{i'r}X_{\pi_{2}(r)}(aX_{i'}+ bX_{j'})\]
 Comparing both sides we have
 
 \[q_{i'\pi_{2}(r)}= q_{i'r}= q_{j'\pi_{2}(r)}, \qquad \forall r \in \mathcal{J}. \] 
 Moreover the hypothesis $ q_{i'j'}=1$ means that $1=q_{i'i'}=q_{j'i'}$ and $1= q_{i'j'}=q_{j'j'}$. It follows that the $ i' $-th and $ j'$-th row of the matrix $ \mathfrak{q} $ coincide, which is contrary  the Lemma \ref{equal rows}. Similarly  $\alpha(X_{j'})\in \mathbb F^{\times} X_{i'}$  or $\alpha(X_{j'})\in \mathbb F^{\times} X_{j'}$. It follows that $\alpha(X_{i'})\in \mathbb F^{\times} X_{\pi_1 (i')}$ and $\alpha(X_{j'})\in \mathbb F^{\times} X_{\pi_1 (j')}$ for some  permutation $ \pi_{1} $ on the set $ \{1,2\}$. 
Now we claim that $ \pi_{2}$ is identity. If  not  from the relations $ X_{r}X_{s} = q_{rs} X_{s}X_{r}$ where $ r,s \in \mathcal{J}$ we have $\mathfrak{m}_{\pi_{2}}\mathfrak{q'} = \mathfrak{q'} \mathfrak{m_{\pi_{2}}}$.\\
 Applying $ \alpha $ to the relations \[ X_{i'}X_{r}= q_{i'r} X_{r}X_{i'},\qquad \forall r \in \mathcal{J} \] and \[X_{j'}X_{r}= q_{j'r} X_{r}X_{j'}, \qquad \forall r \in \mathcal{J} \]    we have $q_{i'r}=q_{\pi_{1}(i')\pi_{2}(r)}$, $q_{j'r}=q_{\pi_{1}(j')\pi_{2}(r)}$ respectively $ \forall r \in \mathcal{J}$ . But this contradicts the theorem  hypothesis. Thus $ \pi_{2}$ is an identity permutation. \\\\
 If  $\alpha(X_{i'})\in \mathbb F^{\times} X_{j'}$ and $\alpha(X_{j'})\in \mathbb F^{\times} X_{i'}$ then applying $\alpha $ to the relation \[ X_{i'}X_{r}= q_{i'r} X_{r}X_{i'}, \qquad \forall  r \in \mathcal{J} \] we have  $ q_{i'r}=q_{j'r} $   $ \forall  r \in \mathcal{J}$ which  contradicting the 
 hypothesis $ E=0 $ in view of Lemma \ref{equal rows}.
\end{proof}

 \section{Proof of Theorem 2}  


The following fact shown in \cite{OP1995} reduces the question of automorphisms to the case where the group $\Lambda$ (Definition \ref{lam-grp}) is torsion-free.   
\begin{lem}[\cite{OP1995}]\label{Lbda_crul_cse}
Let $p$ denote the size of the torsion subgroup of $\Lambda$. The subalgebra $\widehat{\mathcal O'}$ of $\widehat{\mathcal O}_{\mathfrak q}$ generated by the powers $X_i^{\pm p}$ of the indeterminates $X_i$ is a characteristic sub-algebra of the same rank. Moreover $\widehat {\mathcal O}_{\mathfrak q}$ is free left $\widehat{\mathcal O'}$-module of finite rank and the corresponding $\lambda$-group $\Lambda'$ associated with $\widehat{\mathcal O'}$ is torsion free.
\end{lem}

We may thus assume that $\Lambda \cong \mathbb Z^l$ for some natural number $l$. Fixing a $\mathbb Z$-basis $p_1, \cdots, p_l$ in $\Lambda$ we have    
 \begin{equation}\label{e_i-forms}
     \lambda(\gamma, \gamma') = p_1^{e_1(\gamma,\gamma')}p_2^{e_2(\gamma, \gamma')}\cdots p_l^{e_l(\gamma, \gamma')}, \qquad  \gamma, \gamma' \in \Gamma. 
\end{equation}
\begin{notation}\label{not1}
In view of $\eqref{e_i-forms}$ let  
\begin{equation} \label{rel-matrx}
    \lambda(\mathbf e_i, \mathbf e_j) = p_1^{ m_{(ij), 1}} \cdots  p_l^{m_{(ij), l}} , \qquad   1 \le i < j \le n,  
\end{equation} 
where $\mathbf e_i, \mathbf e_j$ are standard basis vectors of the free $\mathbb Z$-module $\Gamma$. On the $\binom{n}{2}$ pairs $(ij),\ (i < j)$ we assume the lexicographic order. Let $\mathsf M \in \mathrm{Mat}_{\binom{n}{2} \times l}(\mathbb Z)$ be the 
matrix whose $((ij), s)$ entry is the exponent $m_{(ij), s}$ of $p_s$ in \eqref{rel-matrx} ($s = 1, \cdots, l$). 
\end{notation}
 
 We recall that for a given matrix $A \in \gl(n, \mathbb Z)$ the exterior square $\wedge^2 A$ of $A$ is the $\binom{n}{2} \times \binom{n}{2}$-matrix whose rows and columns are indexed by the pairs $(ij)\ (1 \le i < j \le n)$ ordered lexicographically and whose $((ij), (kl))$ entry is the $2 \times 2$-minor corresponding to rows $i,j$ and columns $k,l$.  With $\mathsf M$ as defined above we have the following.
 
\begin{prop_A}
Set $N  =  \binom{n}{2}$ and let $\mathsf M$ be as defined in Notation \ref{not1} above.  Then 
\[  \aut (\mathbb Z^n, \lambda) = \bigl ( \stab_{\gl(n, \mathbb Z)}(\mathsf M) \bigr )^t  \] 
where $t$ denotes transposition and $\stab_{\gl(n, \mathbb Z)}(\mathsf M)$ the stabilizer of $\mathsf M$ in $\gl(n, \mathbb Z)$ with respect to the bivector representation 
\begin{equation*} \displaystyle\extp^2 : \gl(n, \mathbb Z) \rightarrow \gl(N, \mathbb Z), \qquad \  A \rightarrow \wedge^2 A 
\end{equation*}  
of $\gl(n, \mathbb Z)$, that is, 
\[ \stab_{\gl(n, \mathbb Z)}(\mathsf M) = \{ A \in \gl(n,Z) \mid (\wedge^2 A)\mathsf M = \mathsf M  \}. \] 
\end{prop_A}

\begin{proof}
Writing the group $\Lambda \le \mathbb F^{\ast}$ additively, in view of Notation \ref{not1} we have   
\begin{equation}\label{expn_in_gens}
    \lambda(\mathbf e_i, \mathbf e_j) = \sum_{s = 1}^l {m_{(ij),s}} p_s, \qquad\ \ \ \forall 1 \le i < j \le n,
\end{equation} 
where $m_{
(ij),s} \in \mathbb Z$. 
Now let \[ A = (a_{ij}) \in \gl(n, \mathbb Z) \] be such that $A^t \in \aut(\mathbb Z^n, \lambda)$.
Setting  \[ \mathbf e_j' = A^t \mathbf e_j=  \sum_{t = 1}^{n} a_{jt}\mathbf e_t  \]
we note that since $\lambda$ is an alternating function therefore $\lambda(\mathbf e_i', \mathbf e_j')$ may be expressed as follows:
\begin{equation}\label{transf_commtr}
    \lambda(\mathbf e_i', \mathbf e_j') = \sum_{(uv)} a_{(ij), (uv)}\lambda(\mathbf e_u, \mathbf e_v),  
\end{equation} 
where the coefficients appearing in the RHS of the above expression constitute row (ij) of the matrix $\wedge^2 A$. 
Since $A^t$ is $\lambda$-preserving, by \eqref{form_prsvng} we have
 \[ \lambda(\mathbf e_i', \mathbf e_j') = \lambda(\mathbf e_i, \mathbf e_j) \qquad\ \ \forall  1 \le i < j \le n. \] 
Expanding and comparing the coefficients of $p_s \ (s = 1, \cdots, l)$ in both sides of the last equation we get 
\begin{equation} \label{coeff_compr}
    \sum_{uv} a_{(ij), (uv)} m_{(uv),s}  =  m_{(ij),s} \qquad\forall   s = 1, \cdots, l.
\end{equation} 
It follows that
\[ \wedge^2(A)\mathsf M = \mathsf M. \]
Clearly the above reasoning is
 reversible. This establishes the assertion of the theorem.
\end{proof}

\begin{notation}\label{not2}
As before, let $\{\mathbf e_1, \cdots, \mathbf e_n\}$ denote the standard basis in $\Gamma = \mathbb Z^n$. Then $\{\mathbf e_i \wedge \mathbf e_j \mid 1 \le   i < j    \le n \}$ is a basis in $\wedge^2 \Gamma$. As usual, for a permutation $\pi \in S_n$ let $P \in \gl(n, \mathbb Z)$ denote the corresponding permutation matrix. 
Clearly $S_n$ acts on the $\mathbb Z$-module $\wedge^2 \Gamma$ via \end{notation}
\begin{equation}\label{actn-Sn}
\pi(\mathbf e_i \wedge \mathbf e_j) = \wedge^2 P (\mathbf e_i \wedge \mathbf e_j) = \mathbf e_{\pi(i)} \wedge \mathbf e_{\pi(j)}, \qquad \forall \pi \in S_n.    
\end{equation}  

By restriction we get an action of $S_n$ on the subset \[ \bar B =  
\{ \epsilon \mathbf e_i \wedge \mathbf e_j \mid  i < j \ \mathrm{and} \ \epsilon \in \{-1, 1\} \}. \] 

Restricting the above action of $S_n$ to $C_\pi := \langle \pi \rangle$ we consider the $C_\pi$-orbits.  
\begin{defn}\label{orb-sum}
We define an \emph{orbit-sum} $\mathscr O_{ij} \ (1 \le i \ne j \le n)$  as 
\[ \mathscr O_{ij} =  \wedge^2 P (\mathbf e_i \wedge \mathbf e_j) + (\wedge^2 P)^2 (\mathbf e_i \wedge \mathbf e_j) + \cdots + (\wedge^{2} P)^{m}(\mathbf e_i \wedge \mathbf e_j), \]
where $m$ stands for the order of $\wedge^2 P$.   
\end{defn}

\begin{lem} \label{fxd-sbmod-per}
Let $\pi \in S_n \ (n \ge 3)$ be a non-identity permutation and let $P$ denote the corresponding permutation matrix. Let  
$\mathrm{Fix}(\wedge^2 P)$ stand for the sub-module
of $\wedge^2 \Gamma = \mathbb Z^N \ (N = \binom{n}{2})$ left fixed by the $\mathbb Z$-linear map $\wedge^2 P \in \gl(N,\mathbb Z)$.
Then \[ \rk(\mathrm{Fix}(\wedge^2 P)) < \binom{n - 1}{2} + 1. \] 
\end{lem}
\begin{proof}
We begin by observing that the submodule $\mathrm{Fix}(\wedge^2 P)$ is generated by the orbit-sums $\mathscr O_{ij}$ of Definition \ref{orb-sum} where we note that either  $\mathscr O_{ij} = -  \mathscr O_{ji}$ or else 
$\mathscr O_{ij} =   \mathscr O_{ji} = 0$.
In other words if $w \in \mathrm{Fix}(\wedge^2 P)$ then 
\[ w = \sum_{ij} \gamma_{ij} \mathscr O_{ij}, \qquad \gamma_{ij} \in \mathbb Z. \]     

Denoting by $\mathcal N_\pi$ the number of $C_{\pi}$-orbits it follows from the above that $\mathrm{Fix}(\wedge^2 P)$ is generated by at most $\floor{\frac{N_\pi}{2}}$ orbit sums $\mathscr O_{ij}$ and thus 
\begin{equation} \label{rnk-bnd}
\rk(\mathrm{Fix}(\wedge^2 P)) \le  \frac{N_\pi}{2}. 
\end{equation}

Evidently for any non-identity permutation $\pi \in S_n$ the number of fixed points $\mathrm{Fix}(\pi)$ is bounded above by $n - 2$ and this bound in attained only by a transposition $(ij)$. Using this fact and the Burnside formula an upper bound for $\mathcal N_\pi$ was obtained in \cite{Sil} as follows  
\begin{align*} \mathcal N_\pi &=  \frac{1}{\abs {C_\pi}}\Biggl ( n(n - 1) + \sum_{\phi \in C_\pi , 
\phi \ne 1}2\binom{\mathrm{Fix}(\pi)}{2} \Biggr ) \le \frac{ n(n - 1)}{\abs{C_\pi}} +   \frac{(\abs{C_\pi} - 1)}{\abs{C_\pi}}2\binom{n - 2}{2} \\ &= 
(n - 2)(n -3) + \frac{4n -6}{\abs{C_\pi}} \le (n - 2)(n -3) + \frac{4n -6}{2} = n^2 - 3n + 3. 
\end{align*} 
Noting \eqref{rnk-bnd} We thus obtain 
 \[ \rk(\mathrm{Fix}(\wedge^2 P)  \le \frac{N_\pi}{2} =  \frac{n^2 - 3n + 3}{2} < \frac{n^2 - 3n + 4}{2}  = \binom{n - 1}{2} + 1. \]

\end{proof}

\begin{lem}\label{sbmd-fxd-by-M} For a quantum torus $\widehat{\mathcal O}_{\mathfrak q}$ suppose that the  $\lambda$-group $\Lambda$ is torsion-free with rank equal to $l$.  
Let $\mathsf M$ be the matrix defined in Notation \ref{not1} (with respect to some choice of a basis in the $\lambda$-group). The $l$ columns of $\mathsf M$ generate a free submodule of $\wedge^2 \Gamma = \mathbb Z^{\binom{n}{2}}$ of rank $l$. 
\end{lem}
\begin{proof}
We recall that for a matrix over a commutative ring $R$ the \emph{row rank} is defined as the maximum number of linearly independent rows of $R$ and \emph{column rank} has a parallel definition. If $R$ is a domain then these two ranks coincide (e.g., \cite[Chapter 4, Corollary 2.29]{AW}).   
The choice of a basis $\{p_1, \cdots, p_l \}$ in $\Lambda$ induces an isomorphism $\Lambda \cong \mathbb Z^l$. Since \[ \{ \lambda(\mathbf e_i, \mathbf e_j)  \mid 1 \le i < j \le n \}  \] is a generating set for the $\mathbb Z$-module $\Lambda$ therefore the rows of $\mathsf M$ generate $\Gamma_l : = \mathbb 
Z^l$. Let $s \le l$ be the row-rank of $\mathsf M$ and let $U$ be the free submodule spanned by some $s$ linearly independent rows. We note that $\Gamma_l/U$ is a torsion $\mathbb Z$-module and hence finite implying that 
\[ l = \rk(\Gamma_l) = \rk(U) = s.\] 
 The assertion in the lemma is now immediate.
\end{proof}

\begin{lem} \label{simpl-crit}
A quantum torus $\widehat{\mathcal O}_{\mathfrak q}$ whose $\lambda$-group is torsion-free and has rank at least $\binom{n -1}{2} + 1$ has center equal to $\mathbb F$. 
\end{lem}

\begin{proof}
 To this end we recall (e.g., \cite{MP}) that a quantum torus $\widehat{\mathcal O}_{\mathfrak q}$ has the structure of a twisted group algebra $\mathbb F \ast \Gamma$ for $\Gamma = \mathbb Z^n$. By \cite[Lemma 1.1(i)]{OP1995} the center of such an algebra is itself a twisted group algebra $\mathbb F \ast Z$ for a subgroup $Z \le \Gamma$. 

We claim that $\Gamma/ Z$ is torsion-free. Indeed let $\gamma \in \Gamma$ be such that $k \gamma \in Z$ for some $k \in \mathbb N$. In view of \eqref{lmbdadefn} \[  \lambda(k\gamma, \gamma') = \lambda(\gamma, \gamma')^k = 1 \qquad \forall \gamma' \in \Gamma. \] 
But as the group $\Lambda$ is torsion free (by hypothesis) therefore $\lambda(\gamma, \gamma') = 1$. Thus $\gamma \in Z$ and it follows that $\Gamma / Z$ is torsion-free. 
 Let $\rho := \rk(\Gamma/Z)$.
 As the map $\lambda$ of \eqref{lmbdadefn} is constant on the cosets of $Z$ in $\Gamma$ it induces an alternating bicharacter $\bar \lambda$ on $\Gamma/Z$ such that
 \[\bar{\lambda}(\gamma +  Z, \gamma' + Z)  = \lambda(\gamma, \gamma'), \qquad \forall \gamma, \gamma' \in \Gamma.  \]
 
 As the group $\Lambda$ is generated by the  \[q_{ij} = \lambda(\mathbf e_i, \mathbf e_j) =
 \bar {\lambda}(\mathbf e_i + Z, \mathbf e_j + Z ), \]  
 letting $u_1 + Z, \cdots u_\rho + Z$ be a basis of $\Gamma/Z$ we must have \[ \Lambda \le \langle \bar{\lambda}(u_r + Z, u_s + Z) \mid 1 \le r < s \le \rho \rangle = : \Lambda_1. \]  
 Clearly the group $\Lambda_1$ has rank at most $\binom {\rho}{2} = \binom{n - \rk(Z)}{2}$. If $Z$ is nontrivial we thus obtain  $\rk(\Lambda) \le \binom{n - 1}{2}$. But $\rk(\Lambda) = \binom{n - 1}{2} + 1$ by the hypothesis. It follows that $\rk(Z) = 0$ and consequently the center of $\widehat{\mathcal O}_{\mathfrak q}$ must be $\mathbb F$. 
\end{proof}

\begin{prop}\label{non-trivial-autog}
Let $\mathfrak q$ be a multiplicatively antisymmetric matrix such that $q_{kj} = 1 \ (j = 1, \cdots, n)$. Then $\aut_{\mathbb F}(\mathcal O_{\mathfrak q})= $ contains the group $(\mathbb F^\ast)^n \mathbb F^+ \rtimes \mathbb F^+$. 
\end{prop}

\begin{proof}
 Let $ b\in \mathbb F$. We define a map  $ \phi_{b} : \mathcal O_{\mathfrak q} \rightarrow \mathcal O_{\mathfrak q}$ via \[\phi_b(X_{i})=  \begin{cases}
                     X_i  &\text{if } i \neq k\\
                     X_i + b, & \text{if } i=k\end{cases}\]
It is easily checked that  $\phi_{b} \in  \aut_{\mathbb F}(\mathcal O_{\mathfrak q})$ and that $b \mapsto \phi_b$ is an embedding of $\mathbb F^+$ into $\aut_{\mathbb F}(\mathcal O_{\mathfrak q})$. If $\tau \in \aut_{\mathbb F}(\mathcal O_{\mathfrak q})$ defined by $\tau(X_i) = t_iX_i$ is a scalar automorphism in $(\mathbb F^\ast)^n$ then it easy to check that 
\[ \tau^{-1}\phi_{b}\tau=\phi_{t_k b_1}. \] 
In other words the image of $\mathbb F^+$ is normalized by the subgroup $(\mathbb F^\ast)^n$ of scalar automorphisms. The assertion of the proposition now follows. 
\end{proof}
\begin{theorem_2}
A quantum affine space $\mathcal O_{\mathfrak q} = \mathcal O_{\mathfrak q}(\mathbb F^n) \ (n \ge 3)$ whose $\lambda$-group is torsion-free and has rank no smaller than $\binom{n - 1}{2} + 1$ satisfies  
 \[ \aut_{\mathbb F}(\mathcal O_{\mathfrak q})= (\mathbb F^\ast)^n. \]  
Moreover, for each $n \ge 3$ and each $r  < \binom{n - 1}{2} + 1$ there exists an algebra $\mathcal O_{\mathfrak q}$ with $\lambda$-group equal to $\mathbb Z^r$ but whose automorphism group embeds $ (\mathbb F^\ast)^n \rtimes \mathbb F^+$. 
\end{theorem_2}    

\begin{proof}
By Lemma \ref{simpl-crit} the corresponding quantum torus $\widehat {\mathcal O}_{\mathfrak q}$ has center equal to $\mathbb F$.  Hence by \cite[Propositon 1.5]{OP1995} each $\mathbb F$-automorphism $\sigma$ of $\mathcal O_{\mathfrak q}$ lifts to an $\mathbb F$-automorphism $\hat{\sigma}$ of $\widehat {\mathcal O}_{\mathfrak q}$. Consequently there is permutation $\pi \in S_n$ such that \[ \hat{\sigma}(X_i) = \sigma (X_i) = a_iX_{\pi(i)}, \quad i \in  \{ 1,\cdots, n\}, a_i \in \mathbb F^\ast. \]  
Clearly, the image of $\hat \sigma$ in $\aut(\mathbb Z^n, \lambda)$ under the map in \eqref{actn-Gma} is the permutation matrix $P$ corresponding to a permutation $\pi \in S_n$.   
Since $P^t = P^{-1}$ hence $P^t \in \aut(\mathbb Z^n, \lambda)$ as well. Proposition A now tells us that $P \in \stab_{\gl(n, \mathbb Z)}(\mathsf M)$ where $\mathsf M$ is a relations matrix for $\widehat {\mathcal O}_{\mathfrak q}$ as defined in Notation \ref{not1}. In other words $(\wedge^2 P) \mathsf M = \mathsf M$. Thus $\wedge^2 P$ fixes each column of $\mathsf M$ and since $\wedge^2 P$ is $\mathbb Z$-linear it fixes the submodule $W$ of $\wedge^2 \Gamma$ spanned by the columns of $\mathsf M$.  By Lemma \ref{sbmd-fxd-by-M}  
\[ \rk(W) = \binom{n - 1}{2} + 1. \]
If $\pi$ is a  non-identity permutation this contradicts Lemma \ref{fxd-sbmod-per}. The first part of the theorem now follows. For the second part let $r \in \{ 1, \cdots , \binom{n - 1}{2} \}$. Clearly in this case we can find a (multiplicatively antisymmetric) matrix $\mathfrak q$ such that $q_{1j} = 1 \ (j = 2, \cdots, n)$ and the subgroup $\langle q_{ij} \mid i \ge 2, i < j \le n \rangle$ is torsion free with rank $r$. Noting Proposition \ref{non-trivial-autog} we are done.      
\end{proof}
\section{$\aut_{\mathbb F}(\mathcal O_{\mathfrak q})$ when $\dime (\widehat{\mathcal O}_{\mathfrak q})=1$} 
\begin{thm}
A quantum affine space $\mathcal O_{\mathfrak q}$ such that the corresponding quantum torus $\widehat{\mathcal O}_{\mathfrak q}$ has dimension one has a trivial automorphism group, that is, $\aut_{\mathbb F}(\mathcal O_{\mathfrak q})= (\mathbb F^\ast)^n$.
\end{thm}
\begin{proof}
Viewing the corresponding quantum torus $\widehat {\mathcal O}_{\mathfrak q}$ as a twisted group algebra $\mathbb F \ast \Gamma$ we recall (\cite[Lemma 1.1(i)]{OP1995}) that the center of this algebra has the form $\mathbb F \ast Z$ for a subgroup $Z \le \Gamma$. We recall (Section \ref{dim-of-qtorus}) that the dimension of the algebra $\widehat {\mathcal O}_{\mathfrak q}$ equals the cardinality of a maximal independent system of commuting monomials. 

Clearly by the hypothesis such a subgroup $B$ must have rank one. This easily implies that $Z$ is the trivial subgroup and thus $F \ast \Gamma$ has center equal to $\mathbb F$. By the proposition of \cite[Section 1.3]{MP} such an algebra is simple.     

In this case by \cite[Proposition 3.2]{OP1995} each $\mathbb F$-automorphism of $\mathcal O_{\mathfrak q}$ is of the form $X_i \mapsto a_i X_{\sigma(i)}$ where $a_i \in \mathbb F^{\ast}$ and $\sigma$ is a permutation of the subscripts $\{1, \cdots, n\}$. Furthermore, the permutation $\sigma$ occurs if and only if 
\begin{equation}\label{admperm}
q_{ij} = q_{\sigma(i) \sigma(j)}, \qquad \forall \le 1 \le i < j\le n.       
\end{equation}

We consider the decomposition of $\sigma$ into disjoint cycles. If there is a transposition $(ij)$ in this decomposition then by \eqref{admperm}
\[ q_{ij} = q_{ji} = q_{ij}^{-1} \]
and so $q^{2}_{ij}=1$. It follows that  
\[ [X_{i}^{2},X_{j}]=[X_{i},X_{j}]^{2}=q_{ij}^{2}=1. \]
But this means that $\dim(\widehat {\mathcal O}_{\mathfrak q}) \ge 2$ (para 1) and we thus get a contradiction to the hypothesis.

Now let $(i_1i_2i_3 \cdots i_r)$ be an $r$-cycle ($r \ge 3$) in the decomposition of $\sigma$. In view of \eqref{admperm} we have $q_{i_{r-1} i_{r}} =  q_{i_{r} i_{1}}$ whence \[ 1 = q_{i_{r-1} i_{r}}q_{i_{1} i_{r}} = [X_{i_{r-1}} X_{i_{1}}, X_{i_r}].  \]
But again (in view of para 1) this is a contradiction to the assumption on the dimension of the corresponding quantum torus
$\widehat {\mathcal O}_{\mathfrak q}$.
\end{proof}

For the quantum tori the case of dimension one means that no two independent monomials commute and so this may seem to be a rather restrictive. We may think that such a case can only occur if the $\lambda$-group has a relatively large rank, possibly even the maximal possible rank $\frac{1}{2}n(n-1)$. 

However an example was given in \cite[Section 3.11]{MP} of a quantum tori of rank $4$ that has dimension one and whose $\lambda$-group has rank equal to $5$. Using the results of \cite{GQ06} we will now show that there exist rank $n$ quantum tori having dimension one and whose $\lambda$-group has rank as low as $n$.  

\begin{example}
 Let $\mathbf q$ be the matrix  
 
  \begin{equation}
  \mathbf q =  \begin{pmatrix}
  1 &  {\zeta}^{-1}  & {\mu}^{-1}  & {\eta}^{-1}  \\
  {\zeta} &  1 & {\eta}^{-1}  & {\mu} \\
  {\mu} &  {\eta} & 1 & {\zeta}^{-1}  \\
  {\eta} &  {\mu}^{-1}  & {\zeta} & 1 \\
  \end{pmatrix},      
  \end{equation}
   where $\zeta, \mu, \nu \in \mathbb K^\ast$ are assumed to be multiplicatively independent.  Then the $\lambda$-group of the quantum torus ${\widehat{\mathcal{O}}}_{\mathfrak q} $  has rank $3$  but $\dim \widehat {\mathcal O}_{\mathfrak q} = 1$.
 \end{example}
\begin{proof}
We consider the three alternating forms $ \alpha_{i} \ (i = 1,\cdots,3)$ defined on the $\mathbb Z $-module $\mathbb Z^4$ given by 
\begin{align*}
   \alpha_{1}(x,y) &=x_{2}y_{1} - x_{1}y_{2} + x_{4}y_{3}
- x_{3}y_{4}  \\
 \alpha_{2}(x,y) &=x_{3}y_{1} - x_{1}y_{3} + x_{2}y_{4}
- x_{4}y_{2}  \\
 \alpha_{3}(x,y) &=x_{4}y_{1} - x_{1}y_{4} + x_{3}y_{2}
- x_{2}y_{3}   
\end{align*}
We claim that these three forms have no common isotropic $\mathbb{Z}$-submodule of $\mathbb Z^4$ of rank greater than one  by which we mean a $\mathbb Z$-submodule $B$ with $\rk(B) \ge 2$ such that the restriction of $\alpha_i$ to $B$ is trivial for all $i = 1\cdots 3$.   

Indeed we may regard $\alpha_i$ as a form on the $\mathbb Q$-space $V: = \mathbb Q^4 = \mathbb Z^4 \otimes_{\mathbb Z} \mathbb Q$. If $B$ is an isotropic submodule as in the preceding paragraph then clearly $B \otimes_{\mathbb Z} \mathbb Q$ is $\mathbb Q$-subspace of $\mathbb Q^4$ with dimension at least two that is  a common isotropic subspace for $\alpha_i \ (i = 1\cdots 3)$.    

Let $\mathbb H$ denote the quaternion algebra over $\mathbb Q$. We know that $\mathbb H$ is a division algebra with the usual $\mathbb Q$-basis $\{1, i, j , k \}$. It can be easily checked that the matrix images of $i$, $j$ and $k$ in the regular representation $\mathbb H \rightarrow \End_{\mathbb Q} \mathbb H$ of $\mathbb H$ are the gram matrices $M_i$ of the forms $\alpha_1, \alpha_2$ and $\alpha_3$ respectively. Since $\mathbb H$ is a division algebra the nonzero matrices in the $\mathbb Q$-subspace of $\End_{\mathbb Q} \mathbb H$ spanned by the $M_i\ (i = 1 \cdots 3)$ are  non-singular. Then by \cite[Corollary 4]{GQ06} the alternating forms $\alpha_i \ (i = 1, \cdots, 3)$ on the $\mathbb{Q}$-space ${\mathbb{Q}}^{4}$ have no common isotropic subspace of dimension greater than one.   
Thus there cannot exist a common isotropic submodule $B$ of rank greater than one for the given alternating forms $\alpha_i \ (i = 1, \cdots, 3)$. 
Let $\{\mathbf e_1, \cdots, \mathbf e_4 \}$ be the standard basis elements of $\mathbb Z^4$ and set 
 \[  q_{ij}= \zeta^{\alpha_{1}(\mathbf e_{i}, \mathbf e_{j})}\mu^{\alpha_{2}(\mathbf e_{i}, \mathbf e_{j})}\nu^{\alpha_{3}(\mathbf e_{i}, \mathbf e_{j})}, \qquad \forall 1\leq i < j \leq 4.  \] 
	Then $\mathbf q = (q_{ij})$ and the commutator map $\lambda$ (Section\ref{sec2.2}) of the quantum torus $\widehat {\mathcal O}_{\mathfrak q}$ has the form 
	\[ \lambda(\gamma, \gamma') = \zeta^{\alpha_{1}(\gamma,\gamma')}\mu^{\alpha_{2}(\gamma,\gamma')}\nu^{\alpha_{3}(\gamma,\gamma')}, \qquad \forall 1\leq i < j \leq 4.   \] 
 
 If there exist two independent commuting monomials $X^{\gamma},X^{\gamma '} \in \widehat {\mathcal O}_{\mathfrak q}$  then clearly $\alpha_{i}(\gamma,\gamma')= 0 $ for all $ i$, which is a contradiction as we have seen that there does not exist common isotropic submodule $B \le \mathbb Z^4$ with $\rk(B) \ge 2$ for the three forms $\alpha_i\ (i = 1, \cdots, 3)$. The theorem now follows in view of Section \ref{dim-of-qtorus}.
\end{proof}
\begin{remark}
Similarly using the (non-associative) octonion algebra over $\mathbb Q$  we may  define a quantum torus $\widehat{\mathcal{O}_{\mathfrak{q}}}$ of rank $8$ and dimension $1$ whose $\lambda$-group has rank $7$. 
\end{remark}
Finally we comment on the general case of quantum tori of rank $n$. It is shown in \cite{BGH, GQ06} that it is always possible to find $n$ alternating forms on  the space $\mathbb Q^n$ for which there is no common isotropic subspace of dimension greater than one. We may follow an approach similar to that of the preceding proposition to come up with examples of dimension one quantum tori of rank $n$ whose $\lambda$-group has rank equal to $n$. This is much smaller than the maximal possible rank $\frac{1}{2}n(n - 1)$.

\section*{Acknowledgements}
The first author thanks the National Board of Higher Mathematics (NBHM) for financial support (DAE/MATH/2015/059). The second author gratefully acknowledges support from an NBHM research award.


\begin{thebibliography}{99}
\bibitem{Sil}L. Silberman Action of symmetric group on the second exterior power. (MathOverflow,2020)
\bibitem{sagemath}Developers, T. SageMath, the Sage Mathematics Software System (Version v9.2). (https://www.sagemath.org,2020)
\bibitem{Grnhl}Greenhill, C. An algorithm for recognising the exterior square of a matrix. {\em Linear And Multilinear Algebra Volume}. \textbf{46} pp. 213-244 (1999)
\bibitem{KPS94}E. Kirkman, C. \& Small, L. A q-analog for the virasoro algebra. {\em Communications In Algebra}. \textbf{22} pp. 3755-3774 (1994)
\bibitem{AC1992}Alev, J. \& Chamarie, M. Derivations et automorphismes de quelques algebras quantiques. {\em Communications In Algebra}. \textbf{20} pp. 1787-1802 (1992)
\bibitem{AG2016}Gupta, A. Representations of the n dimentional quantum torus. {\em Communications In Algebra}. \textbf{44} pp. 3077-3087 (2016)
\bibitem{GQ06}Gow, R. \& Quinlan, R. On the vanishing of subspaces of alternating bilinear forms. {\em Linear And Multilinear Algebra}. \textbf{54} pp. 415-428 (2006)
\bibitem{BGH}J. Buhler, R. \& Harrish, J. Isotropic Subspaces for Skewforms and Maximal Abelian Subgroups of p-Groups. {\em Journal Of Algebra}. \textbf{108} pp. 269-279 (1987)
\bibitem{MY12}Yakimov, M. Rigidity of quantum tori and the Andruskiewitsch–Dumas conjecture. {\em Selecta Mathematica}. \textbf{20} pp. 421-464 (2012)
\bibitem{VA2000}Artamonov, V. Automorphisms of the skew field of rational quantum functions. {\em Sbornik: Mathematics}. \textbf{191} pp. 1749-1771 (2000)
\bibitem{OP1995}J. M. Osborne, D. Derivations of Skew Polynomial Rings. {\em J. Algebra}. \textbf{176} pp. 417-448 (1995)
\bibitem{MP}McConnell, J. \& Pettit, J. Crossed Products and Multiplicative Analogues of Weyl Algebras. {\em J. London Math. Soc.}. \textbf{109} pp. 400-410 (1988)
\bibitem{Nem}Nemitz, W. Transformations Preserving the Grassmannian. {\em Trans. Amer. Math. Soc.}. pp. 47-55 (1963)
\bibitem{HJ}Horn, R. \& Johnson, C. Matrix Analysis. (Cambridge University Press,2013)
\bibitem{AW}Adkins, W. \& Weintraub, S. Algebra. An approach via module theory. ( Graduate Texts in Mathematics, 136. Springer-Verlag,1992)
\bibitem{BrGo}Brown, K. \& Goodearl, K. Lectures on algebraic quantum groups . (Advanced Courses in Mathematics. CRM Barcelona. Birkhäuser Verlag, Basel,2002)
\bibitem{Berg}Berger, M. Geometry II. (Springer-Verlag,1987)
\bibitem{imlelu:oneway}Impagliazzo, R., Levin, L. \& Luby, M. Pseudo-random Generation from One-Way Functions. {\em Proc. 21st STOC}. pp. 12-24 (1989)
\bibitem{komiyo:unipfunc}Kojima, M., Mizuno, S. \& Yoshise, A. A New Continuation Method for Complementarity Problems With Uniform p-Functions. (Tokyo Inst. of Technology, Dept. of Information Sciences,1987)
\bibitem{komiyo:lincomp}Kojima, M., Mizuno, S. \& Yoshise, A. A Polynomial-Time Algorithm For a Class of Linear Complementarity Problems. (Tokyo Inst. of Technology, Dept. of Information Sciences,1987)
\bibitem{NeebKH2008}Neeb, K. On the Classification of Rational Quantum Tori and the Structure of Their Automorphism Groups. {\em Canad. Math. Bull.}. \textbf{51} pp. 261-282 (2008)
\bibitem{ART1997}Artamonov, V. Quantum polynomial algebras. {\em J. Math. Sci.}. \textbf{87} pp. 3441-3462 (1997)
\bibitem{Br2000}Brookes, C. Crossed products and finitely presented groups. {\em Journal Of Group Theory }. \textbf{3} pp. 433-444 (2000)
\bibitem{miyoki:lincomp}Mizuno, S., Yoshise, A. \& Kikuchi, T. Practical Polynomial Time Algorithms for Linear Complementarity Problems. (Tokyo Inst. of Technology, Dept. of Industrial Engineering,1988,4)
\bibitem{moad:quadpro}Monteiro, R. \& Adler, I. Interior Path Following Primal-Dual Algorithms, Part II: Quadratic Programming. (Dept. of Industrial Engineering,1987,8)
\bibitem{Man}Manin, Y. Quantum Groups and Noncommutative Geometry. (Université de Montréal, Centre de Recherches Mathématiques,1988)
\bibitem{ye:intalg}Ye, Y. Interior Algorithms for Linear, Quadratic and Linearly Constrained Convex Programming. (Stanford Univ., Dept. of Engineering–Economic Systems,1987,7)

\end{thebibliography}
\end{document}